\documentclass[10pt]{amsart}
\usepackage[T1]{fontenc}
\usepackage{mathpazo}

\usepackage[bookmarks=true]{hyperref}       

\usepackage{amssymb,latexsym,amsmath,amscd}
\usepackage{mathtools}
\usepackage{commath}
\usepackage{xspace}
\usepackage{enumerate}
\usepackage{graphicx}
\usepackage{vmargin}
\usepackage{todonotes}

\usepackage{dcpic,pictexwd}
\usepackage{tcolorbox}
\usepackage{subcaption}

\DeclareMathOperator{\CR}{cr}

\theoremstyle{plain}
\newtheorem{theorem}{Theorem}[section]
\newtheorem*{theorem*}{Theorem}
\newtheorem{proposition}[theorem]{Proposition}
\newtheorem{corollary}[theorem]{Corollary}
\newtheorem{lemma}[theorem]{Lemma}

\theoremstyle{definition}
\newtheorem{definition}[theorem]{Definition}

\newtheorem{remark}[theorem]{Remark}
\newtheorem{example}[theorem]{Example}

\newcommand{\enm}[1]{\ensuremath{#1}}          
\newcommand{\cal}[1]{\mathcal{#1}}

\renewcommand{\bar}[1]{\overline{#1}}

\newcommand{\CC}{\enm{\mathbb{C}}}

\newcommand{\SF}{\enm{\mathbb{S}}}

\newcommand{\RR}{\enm{\mathbb{R}}}

\newcommand{\HH}{\enm{\mathbb{H}}}

\newcommand{\Cc}{\enm{\cal{C}}}

\renewcommand{\imath}{\sqrt{-1}}

\renewcommand{\phi}{\Phi}
\renewcommand{\theta}{\vartheta}
\renewcommand{\epsilon}{\varepsilon}

\def\8{\infty}

\def\PP{\mathbb{P}}

\def\Re{\mathop{\mathrm{Re}}}
\def\Im{\mathop{\mathrm{Im}}}

\usepackage{color}

\let\oldtocsubsection=\tocsubsection
\renewcommand{\tocsubsection}[2]{\hskip15pt\oldtocsubsection{#1}{#2}}

\begin{document}
\parskip1pt

\title{Minimal Surfaces via Complex Quaternions}

\author[A. Altavilla]{Amedeo Altavilla}\address{Dipartimento di Matematica,
  Universit\`a degli Studi di Bari Aldo Moro, via Edoardo Orabona, 4, 70125,
  Bari, Italia}\email{amedeo.altavilla@uniba.it}

\author[H.-P. Schr\"ocker]{Hans-Peter Schr\"ocker}\address{
University of Innsbruck, Department of Basic Sciences in Engineering Sciences,
Technikerstra\ss e 13, 6020 Innsbruck, Austria}
\email{hans-peter.schroecker@uibk.ac.at }

\author[Z. \v{S}\'ir]{Zbyn\v{e}k \v{S}\'ir}
\address{Charles University, Faculty of Mathematics and Physics, Sokolovsk\'a 83, Prague 186 75, Czech Republic}
\email{zbynek.sir@karlin.mff.cuni.cz}

\author[J. Vr\v{s}ek]{Jan Vr\v{s}ek}
\address{Department of Mathematics, Faculty of Applied Sciences, University of West Bohemia,
Univerzitn\'i 8, 301 00 Plze\v{n}, Czech Republic and
NTIS – New Technologies for the Information Society, Faculty of Applied Sciences, University of West Bohemia,
Univerzitn\'i 8, 301 00 Plze\v{n}, Czech Republic}
\email{vrsekjan@kma.zcu.cz}

\thanks{Amedeo Altavilla was partially supported by PRIN 2022MWPMAB - ``Interactions between Geometric Structures and Function Theories'' and by GNSAGA of INdAM}

\date{\today}

\subjclass[2020]{Primary: 53A10; Secondary: 16H05, 46S05, 65D17}
\keywords{Minimal surface, Pythagorean hodograph curve, complex quaternions, Weierstra\ss-Enneper representation, Enneper surface}
\begin{abstract}
  Minimal surfaces play a fundamental role in differential geometry, with applications spanning physics, material science, and geometric design. In this paper, we explore a novel quaternionic representation of minimal surfaces, drawing an analogy with the well-established theory of Pythagorean Hodograph (PH) curves. By exploiting the algebraic structure of complex quaternions, we introduce a new approach to generating minimal surfaces via quaternionic transformations. This method extends classical Weierstra\ss-Enneper representations and provides insights into the interplay between quaternionic analysis, PH curves, and minimal surface geometry. Additionally, we discuss the role of the Sylvester equation in this framework and demonstrate practical examples, including the construction of Enneper surface patches. The findings open new avenues in computational geometry and geometric modeling, bridging abstract algebraic structures with practical applications in CAD and computer graphics.
\end{abstract}

\maketitle

\section{Introduction}

Minimal surfaces have been studied for a long time and play an important role in differential geometry. The subject connects naturally with several areas of mathematics, including harmonic analysis, complex analysis, and the calculus of variations. A minimal surface in \(\mathbb{R}^3\) is one with zero mean curvature at every point, and this simple condition leads to many interesting and sometimes surprising geometric shapes. A key development in the theory is the Weierstraß–Enneper representation, which describes minimal surfaces using holomorphic functions that satisfy a specific isotropy condition. This approach has led to many explicit examples and has applications in fields such as geometric modeling and materials science~\cite{alarcon21,osserman,Forstneri2023}.

If \(X = X(u,v)\colon \mathbb{C} \simeq \mathbb{R}^{2} \to \mathbb{R}^{3}\) is an isothermal parameterization of a minimal surface, then the complex-valued function \(\Phi\colon \mathbb{C} \to \mathbb{C}^{3}\), defined for \(z = u + \imath v \in \mathbb{C}\) as
\begin{equation*}
\Phi(z) \coloneqq \frac{\partial X}{\partial u}(z) - \imath \frac{\partial X}{\partial v}(z),
\end{equation*}
is holomorphic and satisfies the isotropy condition
\[
\Phi_1(z)^2 + \Phi_2(z)^2 + \Phi_3(z)^2 \equiv 0.
\]
This condition can be addressed via the classical Weierstraß–Enneper factorization:
\begin{equation}
  \label{eq:weierstrass-enneper}
  \Phi = \frac{1}{2} \left(f(1 - g^2), \imath f(1 + g^2), 2fg\right),
\end{equation}
where \(f\) and \(g\) are a holomorphic and a meromorphic function, respectively, satisfying some compatibility condition.
Note that \(f\) and \(g\) can be recovered from $\Phi$ as
\[
  f = \Phi_1 - \imath \Phi_2,\quad
  g = \frac{\Phi_3}{\Phi_1 - \imath \Phi_2}.
\]

More recently, algebraic approaches to minimal surface theory have emerged. In particular, Odehnal~\cite{odehnal16} explored algebraic minimal surfaces and their symmetries, highlighting the fruitful interplay between algebra and geometry. These efforts parallel developments in the theory of Pythagorean Hodograph (PH) curves~\cite{farouki08}, which are polynomial or rational parametric curves \(\gamma\colon I \subset \mathbb{R} \to \mathbb{R}^{3}\) characterized by their hodographs satisfying the Pythagorean condition
\begin{equation}\label{PH}
x'(t)^2 + y'(t)^2 + z'(t)^2 = \sigma(t)^2,
\end{equation}
for some polynomial or rational function \(\sigma(t)\). Polynomial PH curves are notable for enabling a polynomial arc-length function and polynomial as well as rational PH curves have rational offsets. These are valuable features in applications such as CNC machining and robotics~\cite{arrizabalaga22,Farouki_2016,Farouki_1996,Nittler_2016,otto21,Schraeder_2014,Tsai_2001}.

These curves are usually constructed using quaternionic polynomials: Given \(A(t) = u(t) + v(t)i + p(t)j + q(t)k\), one defines the hodograph of a PH curve by
\[
\gamma'(t) = A(t) i A(t)^c,
\]
where \(A(t)^c \coloneqq u(t) - v(t)i - p(t)j - q(t)k\). This use of algebraic conjugation elegantly encodes spatial rotation and underlines the algebraic power of PH curve theory.

In this paper, we introduce a novel framework that unifies these themes through the lens of complex quaternions \(\mathbb{H}_{\mathbb{C}} \coloneqq \mathbb{H} \otimes \mathbb{C} \cong \mathbb{H} \oplus \imath \mathbb{H}\). Specifically, we demonstrate that any isothermally parametrized minimal surface can be constructed from a fixed null quaternion \(L = i + \imath j\) by conjugation with a complex quaternionic function \(\chi(z)\), as
\begin{equation}
  \label{eq:chichi}
  \Phi(z) = \chi(z) L \chi(z)^{-1}.
\end{equation}
This formula directly mirrors the quaternionic construction of PH curves and extends the analogy to a geometric-algebraic description of minimal surfaces.

This representation is rooted in the algebraic structure of the Sylvester equation over \(\mathbb{H}_\mathbb{C}\), which has been studied in the context of slice regular and semi-regular functions~\cite{altavillaLAA,altavillaCO}. It was shown that for any pair of complex-quaternionic vector functions \(\Theta(z), \Psi(z)\), if they satisfy
\[
\Theta_1(z)^2 + \Theta_2(z)^2 + \Theta_3(z)^2 = \Psi_1(z)^2 + \Psi_2(z)^2 + \Psi_3(z)^2,
\]
then there exists \(\chi(z)\) such that \(\Theta(z) = \chi(z) \Psi(z) \chi(z)^{-1}\).

Our contribution is twofold: First, we provide an algebraic framework for minimal surfaces based on this complex quaternionic formulation. Second, we show how this approach connects naturally to the theory of PH curves, allowing all minimal surfaces to be generated by algebraic ``rotations'' of a prototype vector with vanish complex norm, e.g. \(L\). This leads to a unified representation of minimal surfaces and novel constructions of rational minimal patches—most notably, Enneper-type surfaces defined via algebraic data and quaternionic polynomial preimages of degree one.

To the best of our knowledge, this quaternionic representation of minimal surfaces is new. It extends recent results~\cite{FAROUKI2022127439,Hao2020} and opens the door to a systematic algebraic theory that parallels the PH curve framework. In the following sections, we lay out the foundational tools, introduce the Sylvester equation, and demonstrate the construction of minimal surfaces through examples and applications in geometric modeling. After the technical Section~\ref{sec:toolbox} introducing the primary tools—quaternions, complex quaternions, PH curves, and minimal surfaces—we will explore the role of the Sylvester equation in the context of complex quaternions in Section~\ref{sec:sylvester}. We will then examine the construction of minimal surfaces via this quaternionic framework and provide concrete examples (Section~\ref{sec:minimal}). The quaternionic setting lends itself particularly to the construction of \emph{polynomial} and \emph{rational} minimal surfaces. While the former appear now and then in literature, not much seems to be known about the latter (cf. \cite{odehnal16} for some examples). In Section~\ref{sec:enneper-patches} we re-visit the famous Enneper surface and, more in the spirit of computer aided surface design, establish a general result concerning the reconstruction of Enneper surface patches given an orthogonal basis at the vertices of a parameter rectangle.

\section{Toolbox}
\label{sec:toolbox}

In this section we give a short introduction on the main tools we will require later. We state definitions, notations and relevant results on minimal surfaces and their Weierstra\ss-Enneper representation, on quaternions and complex quaternions, and on PH curves.

\subsection{Minimal Surfaces and the Weierstra\ss-Enneper Representation}
\label{sec:min}

For this part we refer mainly to~\cite{osserman} or, for a more recent account, to~\cite{alarcon21,Forstneri2023}.
Let $X=X(u,v)\colon U\subset\RR^{2}\to\RR^{3}$ be a parametric surface of class $\Cc^{\infty}$, and let $X_{u}, X_{v}$ denote the partial derivatives of $X$ with respect to $u$ and $v$, respectively. We recall that the first fundamental form of $X$ is defined as the $2\times 2$ matrix \[\begin{pmatrix} E&F\\F&G \end{pmatrix},\] where $E=\langle X_{u},X_{u}\rangle, F=\langle X_{u},X_{v}\rangle$ and $G=\langle X_{v},X_{v}\rangle$. A parametric surface $X\colon U\subset\RR^{2}\to\RR^{3}$ is said to be \emph{isothermal} if \[E=G>0\quad\text{and}\quad F=0.\]

Given the parametrization $X$, the unit normal vector to the surface
can be written as
\[
N=\frac{X_{u}\times X_{v}}{|X_{u}\times X_{v}|}.
\]
Knowing $N$, it is possible to define the symbols $e$, $f$ and $g$ as follows
\[
e \coloneqq -\langle X_{uu},N\rangle,\quad f \coloneqq -\langle X_{uv},N\rangle,\quad g \coloneqq -\langle X_{vv},N\rangle.
\]
They define the second fundamental form and the shape operator:
\[
\begin{pmatrix} e&f\\f&g \end{pmatrix},\quad
-\begin{pmatrix} e&f\\f&g \end{pmatrix}
\cdot \begin{pmatrix} E&F\\F&G \end{pmatrix}^{-1}.
\]
The eigenvalues $k_{1}$, $k_{2}$ of the shape operator are the two principal curvatures of the surface and define the Gaussian curvature $K = k_{1}k_{2}$ as well as the mean curvature $H=\frac{1}{2}(k_{1}+k_{2})$.

\begin{definition}
A smooth parametric surface in $\mathbb{R}^{3}$ is a minimal surface if its mean curvature $H$ vanishes at every point.
\end{definition}

For an isothermal parametric surface it holds that $X_{uu}+X_{vv}=(2EH) \cdot N$. Therefore, an isothermal $X$ represent a minimal surface if and only if $X$ is harmonic (i.e. in the kernel of the usual Laplacian). At this point it is natural to link the theory of minimal surfaces to that of harmonic and hence also to that of holomorphic functions~\cite{osserman}.

Given a parametric surface $X$, consider the complex function
\begin{equation}
  \label{eq:Rder}\Phi \coloneqq X_{u}-\imath X_{v}.
  \end{equation}
Then, we have that
$X$ is isothermal if and only if
\[\langle \Phi,\Phi\rangle=\Phi_{1}^{2}+\Phi_{2}^{2}+\Phi_{3}^{2} \equiv 0.\]
Moreover, if this condition is satisfied, then $X$ is minimal if and only if $\Phi$ is a holomorphic curve.
On the other hand, if $\Phi\colon U\subset\CC\to\CC^{3}$ is a holomorphic curve, and $\gamma$ is a simple curve from $z_{0}=u_{0}+\imath v_{0}$ to $z=u+\imath v$ in a simply connected component of $U$, then the map $X\colon U\subset\RR^{2}\to\RR^{3}$ defined as
\begin{equation}
  \label{eq:Rint}
  X(u,v) \coloneqq \mathbf c+\Re\Bigl(\int_{\gamma}\Phi(\omega) \dif\omega\Bigr),
\end{equation}
with fixed $\mathbf c\in\RR^3$, defines an isothermal minimal surface. The same is true if we take the imaginary part of the integral in place of the real part. Hence, the theory of minimal surfaces in $\RR^{3}$ is fully encoded in that of holomorphic curves in $\CC^{3}$ with null complex norm.

Traditionally, $\Phi$ is obtained from a meromorphic function $g$ and an analytic function $f$ such that \eqref{eq:weierstrass-enneper} is analytic \cite[Lemma~8.1]{osserman}. This is the famous Weierstra\ss-Enneper formula.

\subsection{Quaternions and Complex Quaternions}

Let $\HH \coloneqq \{q=q_{0}+q_{1}i+q_{2}j+q_{3}k \mid q_{0}, q_{1}, q_{2}, q_{3} \in\RR\}$ where
\begin{equation}
\label{eq:ijk}
i^{2}=j^{2}=k^{2}=-1,\quad ij=-ji=k
\end{equation}
is the usual algebra of real quaternions. Recall that $i$, $j$, and $k$ are not the only imaginary units in $\HH$, but there is a whole sphere of them:
\begin{equation*}
\SF \coloneqq \{q\in\HH\,|\,q^{2}=-1\}=\{q_{1}i+q_{2}j+q_{3}k\,|\,q_{1}^{2}+q_{2}^{2}+q_{3}^{2}=1\}.
\end{equation*}

Any quaternion $q=q_{0}+q_{1}i+q_{2}j+q_{3}k\in\HH$ can be split into its \emph{scalar} and \emph{vector} parts as $q=q_{0}+q_{v}$, where $q_{v}=q_{1}i+q_{2}j+q_{3}k$. With this notation, the product of two quaternions $q=q_0+q_v$ and $p=p_0+p_v$ can be written in the following geometric way:
\[
qp=q_{0}p_{0}-\langle q_{v},p_{v}\rangle+q_{0}p_{v}+p_{0}q_{v}+q_{v}\times p_{v},
\]
where $\langle \cdot,\cdot\rangle$ and $\times$ denote the usual Euclidean scalar and cross products in $\RR^{3}$, respectively. The Euclidean norm of $q$ can be computed as $|q|=\sqrt{qq^{c}}$ where the superscript ``$c$'' denotes the usual conjugation of quaternions. If $q\neq0$, then $q^{-1}=q^{c}/|q|^{2}$.

It is well known that given two quaternions $q=q_{0}+q_{v},p=p_{0}+p_{v}$, there exists $A\in\HH$ such that $q=ApA^{-1}$ if and only if
$q_{0}=p_{0}$ and $|q_{v}|=|p_{v}|$. It is possible to study these condition by means of the so called Sylvester equation $p\xi+\xi q=r$ (see~\cite{BOLOTNIKOV2015567,Bolotnikov2016} for a recent account on this topic).
In particular, given $A$ and $p$, the quaternion $ApA^{-1}$ has the same scalar part as $p$ and its vector part is rotated.
This description of spatial rotations has many computational advantages; details on this and its link to a representation of $\operatorname{SO}(3)$ can be found in several monographs, see e.g.~\cite{farouki08} (or even the  dedicated Wikipedia page \url{https://en.wikipedia.org/wiki/Quaternions_and_spatial_rotation}).

We now pass on to complex quaternions. The algebra of complex quaternions $\HH_{\CC}$ is the complexification of $\HH$, i.e. $\HH_{\CC}=\HH\otimes\CC \cong \HH \oplus \imath \HH$. Elements in $\HH_{\CC}$ have two natural representations underlying the double nature of complex and quaternionic numbers: An element $z\in\HH_{\CC}$
can be written as
\[
z=z_{0}+z_{1}i+z_{2}j+z_{3}k=z_{0}+z_{v}=q+\imath p,
\]
where $i,j,k$ satisfy the usual quaternionic rules \eqref{eq:ijk}, $z_0$, $z_1$, $z_2$, $z_3 \in \CC$ and $q = \Re(z) \coloneqq \Re(z_{0})+\Re(z_{1})i+\Re(z_{2})j+\Re(z_{3})k$, $p= \Im(z) \coloneqq \Im(z_{0})+\Im(z_{1})i+\Im(z_{2})j+\Im(z_{3})k\in\HH$. Therefore, if $z=z_{0}+z_{v}=q+\imath p$ and $w=w_{0}+w_{v}=s+\imath t$ are complex quaternions, then
\begin{equation*}
zw=z_{0}w_{0}-\langle z_{v},w_{v}\rangle+z_{0}w_{v}+w_{0}z_{v}+z_{v}\times w_{v}
\end{equation*}
where $\langle \cdot,\cdot\rangle$ and $\times$ denote the formal generalizations of the usual Euclidean scalar and cross products, i.e., $\langle z_{v},w_{v}\rangle=z_{1}w_{1}+z_{2}w_{2}+z_{3}w_{3}$ and $z_{v}\times w_{v}=(z_{2}w_{3}-z_{3}w_{2})i+(z_{3}w_{1}-z_{1}w_{3})j+(z_{1}w_{2}-z_{2}w_{1})k$.

If $z=z_{0}+z_{v}=q+\imath p\in\HH_{\CC}$, then we can define two natural conjugations coming from the complex and quaternionic structures: We set
\[
\bar z \coloneqq \overline{z_{0}}+\overline{z_{v}}=q-\imath p, \quad
z^{c} \coloneqq z_{0}-z_{v} = q^{c}+\imath p^{c}.
\]
With this, we have that $(\bar z)^{c}=\overline{(z^{c})}$. In particular, it is possible to define the \emph{complex} squared norm of $z\in\HH_{\CC}$ as
\[z^{s} \coloneqq zz^{c} = z_{0}^{2}+z_{v}^{s} = z_{0}^{2}+z_{1}^{2}+z_{2}^{2}+z_{3}^{2}\in\CC,\]
and hence, if $z^{s}\neq 0$, then $z^{-1}=z^{c}/z^{s}$. If $z \in \HH$, then $z^s = \vert z \vert^2$.

\subsection{PH Curves and Minimal Surfaces}

Let $I \subset \RR$ be an interval and $\gamma\colon I\to\mathbb{R}^{2}$, $t \mapsto (x(t),y(t))$ be a smooth parametric curve. As already said in the introduction, Pythagorean-hodograph (PH) curves are rational or polynomial curves whose parametric speed is a polynomial or rational function.

Planar PH curves, first introduced in \cite{farouki90c}, are hence characterized by the equation \[x'^2+y'^2=\sigma^2\] where the prime denotes derivatives with respect to $t$ and $\sigma$ is a polynomial or rational function. Its solutions over the ring of real polynomials are fully described as
\[x'=w(p^2-q^2), \quad y'=2wpq, \quad \sigma = w(p^2+q^2),\]
where $p$, $q$, $w$ are arbitrary polynomials, called the \emph{preimage}. Quite remarkably, a slight modification of this representation is related to the theory of polynomial minimal surfaces. In \cite{Hao2020} the isotropic curve $\Phi$ is constructed in the form
\begin{equation}
  \label{eq:farouki}
\Phi= w\begin{pmatrix}p^2-q^2\\
2pq\\
\imath(p^2+q^2),
\end{pmatrix}
\end{equation}
 where $p$, $q$, $w$ are arbitrary polynomials with real coefficients. In \cite{FAROUKI2022127439} the authors realize that the three polynomials $p$, $q$, $w$ may have complex coefficients. They also show that the resulting surface has a Pythagorean normal and that its parameterization is a PH preserving mapping provided that $w=1$.

 In the present paper, we will provide a full characterization of \emph{all polynomial and rational} minimal surfaces in isothermal parameterization. Moreover, they will appear as special cases of all minimal surfaces represented using a complex-quaternionic formalism. Pushing further the analogy with the PH curves,
 we will relate minimal surfaces rather to \emph{spatial PH curves}. These curves were first considered in \cite{farouki94a} using a three polynomial preimage, and later in \cite{choi02b}, an algebraic representation using quaternionic polynomials was proposed.
 The basic underlying idea is to consider a fixed unitary vector as a quaternion in $\mathbb{S}$, e.g. the quaternion $i$, and to rotate it via conjugation with a quaternionic polynomial $A(t)=a_0(t) + a_1(t){i} + a_2(t){j} + a_3(t){k}$. This will result in a function $A(t)iA(t)^{c}$ that represent the hodograph of a spatial PH curve.

 In view of recent results both polynomial and rational spatial PH curves can be represented in a unified way. The following result was proven for polynomial curves in \cite{Farouki2002} and for rational ones in \cite{SCHROCKER2023128214}.

\begin{theorem}
  \label{th:PH-curves}
  All spatial polynomial (resp. rational) PH curves $\gamma(t)$ are obtained by integration of
  \begin{equation}
    \label{eq:LAiA}
    \gamma'(t) =  \lambda(t)A(t) i A(t)^c,
  \end{equation}
  where $A(t)\in \HH[t]$ is a quaternion valued polynomial and $\lambda(t)$ is a polynomial (resp. rational) function. In the case of rational $\lambda(t)$ all the residues of \eqref{eq:LAiA} must be zero.
\end{theorem}

We would like the reader to note the formal similarity between Equations~\eqref{eq:chichi} and \eqref{eq:LAiA}, in particular in the case where the quaternionic function $\chi(z)$ of \eqref{eq:chichi} is rational. We will make this analogy more explicit in Section~\ref{sec:minimal}.

It the study of PH curves it is natural to start with the lowest degree case. It comes from a quaternionic preimage $A(t)$ of degree one and $\lambda \equiv 1$. The following remark summarizes some known results but puts them into the context of our presentation.

\begin{remark}\label{th:PH3}
 Polynomial spatial PH \emph{cubics} can be characterized in terms of the legs of their control polygon, see \cite[Proposition 21.1]{farouki94a}, and are known to be curves of constant slope over the planar Tschirnhausen cubic. Indeed, with the exception of the cubically parameterized straight line, all spatial cubic polynomial PH curves are obtained by integrating \eqref{eq:LAiA} with a trivial real factor $\lambda(t)=1$ and a linear quaternionic preimage polynomial $A(t)=A_1t+A_0$. This system of parametric cubic PH curves has 8 degrees of freedom (4 free parameters for each of $A_1, A_0$). There are however following modifications of the preimage which will produce the same curve up to the isometries, scalings, and affine re-parameterizations:
  \begin{itemize}
    \item Multiplication $A(t)(\cos \alpha + i \sin \alpha)$ will not change the curve due to the stabilization of $i$,
      \item reparameterization $t\to (a_1 t + a_0)$ will produce a reparameterized and scaled PH curve,
      \item quaternion multiplication $R\, A(t)$ will rotate and scale the curve.
    \end{itemize}
    This family of preimage modification has 7 degrees of freedom ($\alpha, a_1, a_0$ and 4 coefficients of $R$ and will reduce the 8 dimensional family of cubics to 1 dimensional family considered up to the mentioned transformations. More precisely, it can be shown explicitly that any linear $A(t)$ can be modified to the form $t+i \sin \varphi + j \cos \varphi $ which produces, via formula \eqref{eq:LAiA}, a PH cubic of the constant slope $\varphi$ with the axis $(i \sin \varphi + j \cos \varphi)$ and its projection to a perpendicular plane is a Tschirnhausen cubic.
\end{remark}

The representations \eqref{eq:weierstrass-enneper} and \eqref{eq:farouki} are equivalent in the case of polynomial minimal surfaces. And in fact some analogy between this theory and that of PH curves has been used in the works~\cite{FAROUKI2022127439, KIM2008217, ueda, ueda2} to construct maps that preserve the PH condition: Given a planar PH curve $\gamma$, and a polynomial minimal surface $X$, then $X \circ \gamma$ is a PH curve contained in the surface.

Now, in order to obtain an algebraic representation for minimal surfaces like that in Formula~\eqref{eq:LAiA} we will make use of the algebra of complex quaternions $\HH_{\CC}$ and ideas from the theory of quaternionic function theory. In particular we will exploit results on the Sylvester equation applied to holomorphic curves with values in the algebra $\HH_{\CC}$.

\section{The Sylvester Equation}
\label{sec:sylvester}

In this section, we analyze the so-called \emph{Sylvester Equation} in $\HH_{\CC}$: Given $F$, $G$, $H \in \mathbb{H}_{\mathbb{C}}$, we seek all possible solutions $z \in \mathbb{H}_{\mathbb{C}}$ satisfying
\begin{equation*}
  F z + z G = B.
\end{equation*}
Depending on the context, this equation represents various important operators (see, e.g., \cite{Bhatia1997}). It has also been studied in the framework of quaternionic regular functions, where it arises in the theory of intrinsic holomorphic functions with values in $\mathbb{H}_{\mathbb{C}}$ \cite{altavillaCO,altavillaLAA}.
Here, we provide a concise discussion of this equation in the setting of \( \mathbb{H}_{\mathbb{C}} \).
Our main interest in this equation lies in the particular case in which $B=0$. The analysis of this case will give algebraic conditions on the existence of some $z$, such that $F=-zGz^{-1}$.

First of all, given $F$, $G\in\HH_{\CC}$, we define the Sylvester operator $\mathcal{S}_{F,G}\colon\HH_{\CC}\to\HH_{\CC}$ as $\mathcal{S}_{F,G}(z)=F z+zG$.  It is clear that, if one among $F_{v}$ or $G_{v}$ is equal to zero, then the operator becomes trivial. In fact, if  $F_{v}=0$, then $\mathcal{S}_{F,G}(z)=F z+zG=F_{0}z+zG=z(F_{0}+G)$.
The following result characterizes when $\mathcal{S}_{F,G}$ is non singular in the more general case in which $F_{v}$ and $G_{v}$ are both non-zero.

\begin{theorem}
Let $F=F_{0}+F_{v},G=G_{0}+G_{v}$ be
any pair of elements in $\HH_{\CC}\setminus\{0\}$ such that $F_{v}$, $G_{v}\neq 0$. Then, the Sylvester operator $\mathcal{S}_{F,G}$ is singular if and only if
\[
(F_{0}+G_{0})^{4}+2(F_{0}+G_{0})^{2}(F_{v}^{s}+G_{v}^{s})+(F_{v}^{s}-G_{v}^{s})^{2}=0.
\]
Moreover, the operator has rank $2$ if and only if $F_{0}+G_{0}=F_{v}^{s}-G_{v}^{s}=0$, while it has rank $3$
if and only if $0\neq F_0+G_0=\pm
  \imath(\sqrt{F_v^{s}}\pm\sqrt{G_v^{s}})$.
\end{theorem}

\begin{proof}
By splitting $F$, $G$ and $z$ into their scalar and vector parts, the equation \(F z + zG = 0\) becomes the following equivalent linear system
\begin{equation*}
\begin{cases}
  F_0z_0-\langle F_v,z_v\rangle+G_0z_0-\langle z_v,G_v\rangle=0,\\
  F_0z_v+z_0F_v+F_v\times z_v+z_0G_v+G_0z_v+z_v\times G_v = 0
\end{cases}
\end{equation*}
that is equivalent to
\begin{equation}
  \label{syl2}
  \begin{cases}
    (F_0+G_0)z_0-\langle F_v+G_v,z_v\rangle=0  \\
    (F_v+G_v)z_0+(F_0+G_0)z_v+(F_v-G_v)\times z_v=0.
  \end{cases}
\end{equation}
The matrix $M$ associated to the system is given by
\[M=\begin{pmatrix}
      F_0+G_0  & -F_1-G_1 & -F_2-G_2 & -F_3-G_3 \\
      F_1+G_1  & F_0+G_0 & -(F_3-G_3) & F_2-G_2 \\
      F_2+G_2  & F_3-G_3 & F_0+G_0 & -(F_1-G_1)\\
      F_3+G_3  & -(F_2-G_2) & F_1-G_1 & F_0+G_0
    \end{pmatrix}.
\]
The four (complex) eigenvalues of this matrix can be written as
\[
F_0+G_0\pm\sqrt{\pm 2\sqrt{F_v^{s}}\sqrt{G_v^{s}}-(F_v^s+G_v^s)},
\]
that is
\[
F_0+G_0\pm \imath(\sqrt{F_v^{s}}\pm\sqrt{G_v^{s}}).
\]
Any of these values is equal to zero if and only if their product is, i.e. if and only if $0=\det M=(F_{0}+G_{0})^{4}+2(F_{0}+G_{0})^{2}(F_{v}^{s}+G_{v}^{s})+(F_{v}^{s}-G_{v}^{s})^{2}$.

Now, if one among the four eigenvalues, say $F_0+G_0-\imath(\sqrt{F_v^{s}}-\sqrt{G_v^{s}})$, is zero then either $F_0+G_0=0$ and $\sqrt{F_v^{s}}-\sqrt{G_v^{s}}=0$ which is equivalent to saying that exactly two eigenvalues are zero (i.e. $\mathcal{S}_{F,G}$ has rank two) or $0\neq F_0+G_0= \imath(\sqrt{F_v^{s}}-\sqrt{G_v^{s}}) $, which implies that this is the only zero eigenvalue and $\mathcal{S}_{F,G}$ has rank three.
\end{proof}

We notice that if the components of $F$ and $G$ are real quaternions, then the previous
equality is satisfied if and only if $F_0=-G_0$ and $\|F_v\|=\|G_v\|$. This
is exactly the standard condition for two quaternions to be equivalent.
In this case it is well known that the matrix $M$ has rank equal to two.
In the next example we provide an explicit case in which the rank is equal to three.
\begin{example}
Take $F=F_v=\imath j+k$ and $G=\imath+\sqrt{2}/2(j+k)$. In this case, the matrix $M$ can be written as
\[
  M =
  \begin{pmatrix}
    \imath & 0 & -(\imath+{\sqrt{2}}/{2}) &-(1+\sqrt{2}/2)\\
    0 & \imath &-(1-\sqrt{2}/2) & \imath-\sqrt{2}/2\\
    \imath+\sqrt{2}/2 & 1-\sqrt{2}/2 & \imath & 0\\
    1+\sqrt{2}/2 & -(\imath-\sqrt{2}/2) & 0 & \imath
  \end{pmatrix}
\]
and it is easy to check that $\operatorname{rank}M = 3$.
\end{example}

We now pass to the case of functions from $\CC$ to $\HH_{\CC}$. With a small abuse of notations, these can be written as complex curves
\[f(z)=f_{0}(z)+f_{1}(z)i+f_{2}(z)j+f_{3}(z)k=f_{0}(z)+f_{v}(z)=(f_{0}(z),f_{1}(z),f_{2}(z),f_{3}(z)),\]
where, for any $\ell \in \{0,1,2,3\}$, $f_{\ell}$ represents a complex valued function of a complex variable.

Now, in the ring of meromorphic curves, the same result as described above holds without great issues. This behavior was deeply explained in~\cite{altavillaLAA}, where it is possible to find the following two results. First of all, given a domain $\Omega\subset\CC$ and two holomorphic curves $f$, $g\colon \Omega\to\HH_{\CC}$ we introduce the following Sylvester operator:
\[
\mathcal{S}_{f,g}(\chi) \coloneqq f\chi+\chi g,
\]
where $\chi\colon \Omega\to\HH_{\CC}$ is any other meromorphic curve.
Here the holomorphicity is regarded with respect to the usual complex structure in the domain and the multiplication by $\imath$ in the codomain.

By studying the operator $\mathcal{S}_{f,g}$ it is possible to prove the following result.
\begin{proposition}[Proposition 5.3~\cite{altavillaLAA}]
  \label{prop:chichi}
  Let $\Omega\subset\CC$ be a domain and let $f$, $g\colon \Omega\to\HH_{\CC}$ be holomorphic curves such that $f_{0}=g_{0}$, then there exists $\chi\colon \Omega\to\HH_{\CC}$ such that
  \begin{equation}
    \label{eq:fchigchi}
    f=\chi g\chi^{-1}
  \end{equation}
  if and only if $f_{v}^{s}\equiv g_{v}^{s}$.
\end{proposition}

\begin{remark}
  \label{remsol}
  As explained in~\cite[Theorem 7.1]{altavillaLAA}, in the hypotheses of previous proposition, the set of $\chi$, such that $f=\chi g\chi^{-1}$ is given by \[\chi= fh+hg^{c},\] where $h\colon \Omega\to\HH_{\CC}$ is any meromorphic function such that $\chi$ is invertible. In particular, by varying $h$ it is possible to find many invertible solutions to Equation~\ref{eq:fchigchi}.
\end{remark}

\section{Minimal Surfaces}
\label{sec:minimal}

We are particularly interested in the specific consequences of Proposition~\ref{prop:chichi} for the case of the \emph{constant} function \(g(z) = L\) where \(L = i + \imath j\) (or any other non-zero complex quaternion of zero norm). This constant is such that \(g_0 = 0\) and $g^s_v=g^s=0$. Proposition~\ref{prop:chichi} provides all isotropic curves, and hence also minimal surfaces, for this situation:

\begin{corollary}
  \label{cor:chichi}
  Let $\Omega\subset\CC$ be a domain and $\Phi\colon\Omega\to\HH_{\CC}$ be a holomorphic function such that $\Phi_{0}\equiv \Phi^s \equiv 0$. Then, there exists a meromorphic curve $\chi\colon\Omega\to\HH_{\CC}$ such that
  \begin{equation}
    \label{eq:chichi2}
    \Phi=\chi L\chi^{-1}.
  \end{equation}
  All minimal surfaces in isothermal coordinates can be constructed from \eqref{eq:chichi2} by integration via \eqref{eq:Rint}.
\end{corollary}

In particular, the previous result asserts that---from the algebraic point of view---all ``hodographs'' of minimal surfaces are equivalent under conjugation with a suitable meromorphic curve $\chi$. Therefore, in the algebra $\HH_{\CC}$ of complex quaternions, it makes sense to perform at least formally all the usual computations for PH curves. We will elaborate this idea after an example to motivate the theory and to clarify the algebraic construction.

\begin{example}[Catenoid]
  \label{ex:catenoid}
  \cite[Example~2.7]{Forstneri2023} considers a conformal parametrization of the
  catenoid obtained as the real part of the integral of
  \begin{equation}
    \label{eq:catenoid_Phi}
    \Phi(z) = \frac{z^2-1}{2z^2}i - \frac{\imath(z^2+1)}{2z^2}j - \frac{1}{z}k.
  \end{equation}
  With
  \begin{equation*}
    \chi(z) = \frac{1}{2z^2} ((1-3z^2)i + \imath(1-z^2)j + 2zk)
  \end{equation*}
  we indeed have $\Phi = \chi L \chi^{-1}$ as claimed by Corollary~\ref{cor:chichi}. The thus obtained isothermal parametrization of the catenoid reads
  \begin{equation}
    \label{eq:catenoid}
    X(u,v) =
    \frac{1}{2}
    \begin{pmatrix}
      u + \frac{u}{u^2+v^2} \\
      v + \frac{v}{u^2+v^2} \\
      -\ln(u^2+v^2)
    \end{pmatrix}.
  \end{equation}
  The resulting surface patch together with a more common parameterization of the surface is depicted in Figure~\ref{fig:catenoid}. We remark that the parametrization \eqref{eq:catenoid} is not rational even if \eqref{eq:catenoid_Phi} is and we suggest to compare this with Example~\ref{ex:rational}, below.
\end{example}

\begin{figure}
  \centering
  \includegraphics[]{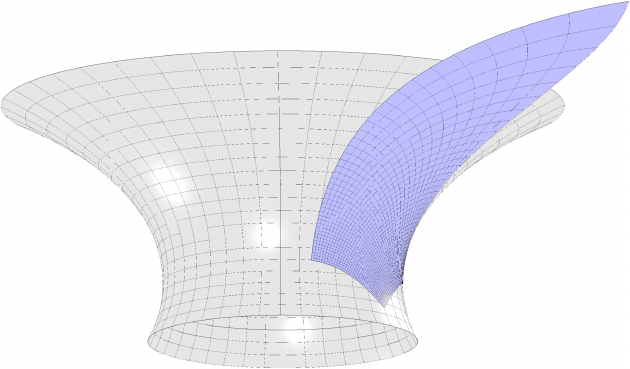}
  \caption{Catenoid in standard and in isothermal parametrization.}
  \label{fig:catenoid}
\end{figure}

We are now going to make the (only seemingly small) step to re-formulate our construction of polynomial and rational minimal surfaces in isothermal parametrization in analogy to PH curves, cf. Theorem~\ref{th:PH-curves}:
\begin{itemize}
\item We replace the meromorphic function \(\chi\) by a quaternionic polynomial~\(A\),
\item we replace the inversion in \eqref{eq:chichi2} by conjugation, and
\item we absorb denominators and possible scalar polynomial factors into a
  rational or polynomial function~\(\lambda\).
\end{itemize}

\begin{theorem}
  \label{th:AA}
  Let $\Omega\subset\CC$ be a simply connected domain. All polynomial (resp. rational) minimal surfaces $X\colon \Omega\to\RR^{3}$ in isothermal parameterization are obtained by the formula \eqref{eq:Rint} from an isotropic complex curve $\Phi$ of the form
  \begin{equation}
    \label{eq:lALA}
    \Phi = \lambda A L A^c,
  \end{equation}
  where $A\in \HH_{\CC}[z]$ is a complex quaternion valued polynomial and $\lambda$ is a complex polynomial (resp. rational) function in $z$. In the case of rational $\lambda$ all the residues of \eqref{eq:lALA} must be zero in order to ensure a rational integral.
\end{theorem}
\begin{proof}
As stated in Remark~\ref{remsol}, the solution to the equation \(\Phi = \chi L \chi^{-1}\) has the form \(\chi = \Phi h + h L\) for some rational quaternionic function $h$ such that \(\chi\) is invertible. Indeed, since \(\Phi^2 = -\langle \Phi, \Phi \rangle = -\Phi^s = 0\) and, similarly, \(L^2 = 0\), we can verify that:
\[
\Phi(\Phi h + h L) - (\Phi h + h L)L = \Phi h L - \Phi h L = 0.
\]

The statement of the theorem follows easily when \(\Phi\) is rational. Writing \(\chi = B/w\) with \(B \in \HH_\CC[z]\) and \(w \in \CC[z]\) we have \(\chi^{-1} = wB^c/B^s\), providing
\[
  \Phi = \chi L \chi^{-1}
  = \frac{B}{w} L \frac{wB^c}{B^s}
  = \frac{BLB^c}{B^s}
  = \lambda B L B^c
\]
with \(\lambda = 1/B^s\) and \(A = B\).

In the special case where \(\Phi\) is a polynomial some more work is needed to show that a \emph{polynomial \(\lambda\)} is possible.

Writing \(\Phi = \Phi_1i + \Phi_2j + \Phi_3k = BLB^c/B^s\) for a suitable \(B \in \HH_\CC[z]\), we set \(\lambda \coloneqq \gcd(\Phi_1,\Phi_2,\Phi_3)\). The quaternionic polynomial \(\Phi' \coloneqq \Phi/\lambda\) then is reduced, that is, it satisfies \(\gcd(\Phi'_1,\Phi'_2,\Phi'_3) = 1\). As above, there exists \(C \in \HH_\CC[z]\) such that \(\Phi' = CLC^c/C^s\). The fact that \(\Phi'\) is polynomial and reduced means precisely that \(C^s\) is the scalar polynomial factor of maximal degree of \(CLC^c\) (the \(\gcd\) of its coefficients). To compute it we write \(C = C_0 + C_1i + C_2j + C_3k\), set
\[
  C_{03} \coloneqq \imath C_0+ C_3,\quad
  C_{12} \coloneqq C_1+\imath C_2,
\]
and obtain
\[
  CLC^c = (C_{12}^2 - C_{03}^2)i - \imath (C_{12}^2 + C_{03}^2)j + 2C_{12}C_{03}k.
\]
Now
\begin{align*}
  C^s = \gcd(C_{12}^2-C_{03}^2,-\imath(C_{12}^2+C_{03}^2), 2C_{12}C_{03}) &= \gcd(C_{12}^2-C_{03}^2,C_{12}^2+C_{03}^2, C_{12}C_{03}) \\
  &= \gcd(C_{12}^2,C_{03}^2, C_{12}C_{03}) \\
  &= \gcd(C_{12},C_{03})^2 = \sigma^2
\end{align*}
where \(\sigma \coloneqq \gcd(C_{12},C_{03})\).

Finally, we show existence of \(h = \frac{1}{2}(\alpha - \beta j) \in \HH_\CC[z]\) such that the polynomial
\[
  A \coloneqq \frac{CLC^c}{C^s}h + hL = \frac{1}{\sigma^2} CLC^ch + hL
\]
satisfies \eqref{eq:lALA}. With \(A\) thus defined we know that \(\Phi' =
(ALA^c)/A^s\) holds true and we need to show that \(h\) can be chosen such that
\(A^s = 1\). We compute explicitly
\begin{align*}
  A^s = AA^c =
  \Bigl(\alpha\frac{C_{12}}{\sigma} + \beta\frac{C_{03}}{\sigma}\Bigr)^2.
\end{align*}
Because \(C_{12}/\sigma\) and \(C_{03}/\sigma\) are relatively prime by
Bézout's identity there exist \(\alpha\), \(\beta \in \CC[z]\) such that \(A^s = 1\).
\end{proof}

\begin{example}
  \label{ex:catenoid2}
  We resume Example~\ref{ex:catenoid}. Theorem~\ref{th:AA} predicts existence of a rational function $\lambda \in \CC(z)$ and a quaternionic polynomial $A \in \HH_\CC[z]$ such that\footnote{Actually, Theorem~\ref{th:AA} predicts existence of $\lambda$ and $A$ only for rational minimal surfaces. But the constructive proof only assumes rational $\Phi$ and is applicable here as well.}
  \begin{equation*}
    \Phi(z) = \lambda(z) A(z) L A(z)^c.
  \end{equation*}
  Indeed, this equation is satisfied with
  \begin{equation}
    \label{eq:Alambda}
    A = (1-3z^2)i + \imath(1-z^2)j + 2zk
    \quad\text{and}\quad
    \lambda = \frac{1}{A^s} = \frac{1}{8z^4}.
  \end{equation}
  Note that $A$ is just the numerator of the rational function $\chi$ introduced in Example~\ref{ex:catenoid}.
\end{example}

The next proposition shows that the representation \eqref{eq:farouki} of \cite{FAROUKI2022127439} and \eqref{eq:lALA} both describe the complete set of polynomial minimal surfaces.

\begin{proposition}
  In the case of \emph{polynomial} minimal surfaces, the representations \eqref{eq:farouki} and \eqref{eq:lALA} are equivalent.
\end{proposition}

\begin{proof}
  Given $p$, $q$, and $w$ as in \eqref{eq:farouki} we can set $\lambda \coloneqq w$ and $A \coloneqq p + pi + qj + qk$. Conversely, given $A = a_0 + a_1i + a_2j + a_3k$ and $\lambda$ as in \eqref{eq:lALA}, we can set $w \coloneqq \lambda$, $p \coloneqq \frac{1}{2}(a_0+a_1+\imath a_2-\imath a_3)$, and $q \coloneqq \frac{1}{2}(\imath a_0 - \imath a_1 + a_2 + a_3)$.
\end{proof}

\begin{example}[Rational Minimal Surface]
  \label{ex:rational}
  We continue Examples~\ref{ex:catenoid} and \ref{ex:catenoid2} but replace the
  integrand $\Phi$ of \eqref{eq:catenoid_Phi} by
  \begin{multline*}
    \frac{z^4+1}{z^2}\Phi = \frac{(z^4+1)(z^2-1)}{2z^4}i - \frac{\imath(z^4+1)(z^2+1)}{2z^4}j - \frac{z^4+1}{z^3}k \\
= \frac{-i-\imath j}{2z^4}
    - \frac{k}{z^3}
    + \frac{i-\imath j}{2z^2}
    - \frac{i+\imath j}{2}
    - kz
    + \frac{(i-\imath j)z^2}{2}.
  \end{multline*}
  Clearly, this can also be written as
  \begin{equation*}
  \frac{z^4+1}{z^2}\Phi = \tilde{\lambda}(z) A(z) L A(z)^c
  \end{equation*}
  where $\tilde{\lambda} = (z^4+1)\lambda z^{-2}$ and both, $\lambda$ and $A$, are as in \eqref{eq:Alambda}. The factor $(z^4+1)z^{-2}$ is chosen such that the residue at $z = 0$ is zero whence the integral is rational:
  \begin{equation*}
    \int \frac{z^4+1}{z^2}\Phi \dif z =
\frac{i+\imath j}{6z^3}
    + \frac{k}{2z^2}
    - \frac{i-\imath j}{2z}
    - \frac{(i-\imath j)z}{2}
    + \frac{kz^2}{2}
    + \frac{(i-\imath j)z^3}{6}.
  \end{equation*}
  Substituting $z = u + \imath v$ and computing the real part, we find the
  minimal surface parametrization
  \begin{equation*}
    X(u,v) =
    \frac{1}{6(u^2+v^2)^3}
    \begin{pmatrix}
      (u^2+v^2+1)u(u^6-u^4v^2-5u^2v^4-3v^6-4u^4-4u^2v^2+u^2-3v^2) \\
      (u^2+v^2+1)(3u^6+5u^4v^2+u^2v^4-v^6+4u^2v^2+4v^4+3u^2-v^2)v \\
      -3(u^2+v^2)(-v+u)(u+v)(u^2+v^2-1)(u^2+v^2+1)
    \end{pmatrix}.
  \end{equation*}
  One part of this surface is depicted in Figure~\ref{fig:rational}.
\end{example}

\begin{figure}
  \centering
  \includegraphics[]{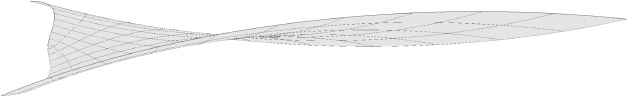}
  \caption{A rational minimal surface.}
  \label{fig:rational}
\end{figure}

\begin{remark}
  \label{rem:reduction-to-real}
  As already indicated in the proof of Theorem~\ref{th:AA}, the polynomial $A \in \HH_\CC$ in Example~\ref{ex:rational} is not unique. In fact, the degrees of freedom related to the stabilizers of $L$ under conjugation with $\HH_\CC$ allow for a quaternionic polynomial $B$ with \emph{real} coefficients such that $ALA^c = BLB^c$. With $A = a_0 + a_ii + a_jj + a_kk$ we have
  \begin{equation}
    \label{eq:realB}
    B \coloneqq \Im a_k + \Re a_0 + (\Re a_i-\Im a_j)i + (\Im a_i+\Re a_j)j + (\Re a_k-\Im a_0)k.
  \end{equation}
\end{remark}

\begin{example}
  Consider the polynomial $A \in \HH_\CC$ of Equation~\eqref{eq:Alambda}. We have $A = a_0 + a_ii + a_jj + a_kk$ with
  \begin{equation*}
    a_0 = 0,\quad a_i = -3z^2+1,\quad a_j = \imath(1-z^2),\quad a_k = 2z.
  \end{equation*}
  By Equation~\ref{eq:realB} the real polynomial \(B \coloneqq -2z^2i + 2zk\) satisfies $ALA^c = BLB^c$.
\end{example}

\begin{remark}
  If $A \in \HH[z]$ and $\lambda \equiv 1$, then the resulting minimal surface has PH curves as parameter lines. This statement is weaker than \cite[Theorem~3]{FAROUKI2022127439} but, in the context of our derivation, rather obvious. It is also clear that the parameter lines are PH curves if $\lambda$ is a square.
\end{remark}

\begin{example}
    Let us demonstrate on this example how efficiently the preimage pair $A$, $\lambda$ can be computed for a given minimal surface. Let us consider the Richmond's surface
    \[X(u,v)=\left(
      \frac{u^3}{3}-uv^2+\frac{u}{u^2+v^2},
      \frac{v^3}{3}-u^2v-\frac{v}{u^2+v^2},
      2 u
    \right),\]
    see e.g. \cite[Section~5]{odehnal16} and Figure~\ref{fig:enneper-richmond}, right. Applying the derivative operator \eqref{eq:Rder} we obtain the holomorphic isotropic curve \[\Phi(z)=\Bigl( \frac{z^4-1}{z^2},\frac{\sqrt{-1} (z^4+1)}{z^2},2\Bigr)\] which we can identify with a complex quaternionic rational function. Any solution to the problem $\chi L \chi^{-1}=\Phi$ has the form $\chi=\Phi h + h L$. We can for example set $h=z^2$ which is the (real) greatest common denominator of $\Phi$ and leads to a polynomial $\chi$
    \[\chi=(z^4+z^2-1)i+\sqrt{-1} (z^4+z^2+1)j+2z^2k.\]
    Setting $A=\chi$ and $\lambda=1/\vert A \vert^2=-1/(4z^2)$ we obtain the desired representation $$X(u,v)=\Re\Bigl(\int_{\gamma}\lambda(z)A(z)LA(z)^c \dif z\Bigr),$$ where the integral is path independent. 
\end{example}

\begin{figure}
  \centering
  \includegraphics[]{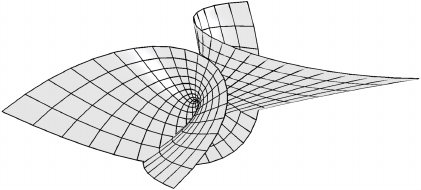}
  \includegraphics[]{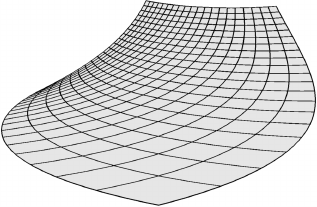}
  \caption{Enneper's minimal surface (left) and Richmond's minimal surface (right) in isothermal parametrization.}
  \label{fig:enneper-richmond}
\end{figure}

\section{Enneper Patches}
\label{sec:enneper-patches}

In this section we will discuss polynomial minimal surfaces which have a quaternionic preimage $A$ of degree one and $\lambda = 1$ in order to demonstrate how our theory can be applied. The resulting polynomial parametrization $X(u,v)$ of the minimal surface is then of degree three whence the minimal surface itself is the Enneper surface \cite{cosin,monterde}. The counterpart of this statement for real PH cubics is their characterization as curves of constant slope over Tschirnhausen cubics. Using the same approach as in Remark~\ref{th:PH3} it can be shown that any linear preimage $A$ generates a cubic minimal surface which is merely a rotation, scaling, translation and reparameterization of the following model example.

\begin{example}
  \label{ex:enneper}
  Consider the linear polynomial $A = z + j$. The minimal surface parametrization obtained as the real part of the integral of $\Phi \coloneqq ALA^c$ reads
  \[
    X(u,v) = (\tfrac{1}{3}u^3-uv^2-u)i + (\tfrac{1}{3}v^3-u^2v-v)j + (v^2-u^2)k.
  \]
  The corresponding surface, parameterized over $[-3, 3] \times [-3, 3]$ is displayed in Figure~\ref{fig:enneper-richmond}, left. It is Enneper's surface in isothermal parametrization.
\end{example}

In \cite{FAROUKI2022127439} the authors have already studied patches of Enneper surfaces and more general minimal surfaces defined by their boundary curves. In contrast to this paper, we rather focus on the interpolation of derivative data.

For actually working with Enneper surface patches over a parameter rectangle with vertices
\begin{equation*}
  (u_0,v_0),\quad (u_1,v_0),\quad (u_1,v_1),\quad (u_0,v_1)
\end{equation*}
it is desirable to design the quadratic polynomial $\Phi(z)$ directly via its values $\varphi_{\ell} \coloneqq \Phi(z_{\ell})$ where
\begin{equation*}
  z_0 = u_0 + \imath v_0,\quad
  z_1 = u_1 + \imath v_0,\quad
  z_2 = u_1 + \imath v_1,\quad
  z_3 = u_0 + \imath v_1.
\end{equation*}
Note that the complex vector $\Phi(z_0)$ is related to the partial
derivatives of $X(u,v)$ at the point $(u_0,v_0)$ via
\begin{equation*}
  \pd{X}{u}(u_0,v_0) = 2\Re\Phi(z_0),\quad
  \pd{X}{v}(u_0,v_0) = -2\Im\Phi(z_0)
\end{equation*}
and similar for $\varphi_1$, $\varphi_2$, and $\varphi_3$. What we actually do describe is therefore the partial derivative vectors at the patch corners. The values $\varphi_0$, $\varphi_1$, $\varphi_2$, and $\varphi_3$ are not independent but subject to the some constraints that are clarified by the following theorem:

\begin{theorem}
  \label{maincub}
  Consider a rectangle $\mathcal{R}\subset\CC$ of legs measuring $r_{1}$, $r_{2}>0$, set $R=r_{1}+ \imath r_{2}$ and denote the rectangle vertices anticlockwise as $P_{0}$, $P_{1}$, $P_{2}$, $P_{3}\in\CC$. Let $\varphi_{0}$, $\varphi_{1}$, $\varphi_{2}$, $\varphi_{3}$ be any four elements in $\HH_{\CC}$. Then, there exists a degree one polynomial $A\in\HH_{\CC}[z]$, such that $\Phi(z)=A(z)LA(z)^{c}$ satisfies $\Phi(P_{\ell})=\varphi_{\ell}$ for $\ell=0,1,2,3$, if and only if the following conditions are satisfied:
  \begin{gather}
    \varphi_\ell + \varphi_\ell^c = \varphi_{\ell}\varphi_\ell^c = 0\quad\text{for $\ell=0,1,2,3;$}\label{eq1}\\
    |R|^{2}\varphi_{0}+R^{2}\varphi_{1}-|R|^{2}\varphi_{2}-R^{2}\varphi_{3}=0;\label{eq2}\\
    \CR(P_0,P_1,P_2,P_3) = \CR(\varphi_0,\varphi_1,\varphi_2,\varphi_3).\label{eq3}
  \end{gather}
  Here, $\CR$ denotes a cross-ratio.
\end{theorem}

Equation~\eqref{eq1} states that the vectors $\varphi_0$, $\varphi_1$, $\varphi_2$, and $\varphi_3$ define points on the projective conic $\mathcal{N}$ of vectorial quaternions of zero norm. The cross ratio on the left-hand side of \eqref{eq3} is the cross ratio of four complex numbers while the cross ratio on the right-hand side is the cross ratio of four points on the conic $\mathcal{N}$. The condition \eqref{eq2} is a typical constraint on the coefficients of the hodograph that arises in similar form also in the context of PH curves \cite[Proposition~21.1]{farouki08}.

We precede the proof of Theorem~\ref{maincub} by a remark that describes convenient coordinates for our investigation and a technical lemma.

\begin{remark}
  \label{changeofvar}
  Given a degree one polynomial $B(z)=(1-z)B_{0}+zB_{1}\in\HH_{\CC}[z]$ we want to find convenient coordinates to evaluate it at the vertices $P_{0}, P_{1}, P_{2}, P_{3}$ of a generic rectangle $\mathcal{R}$ of legs measuring $r_{1},r_{2}>0$. Set $R=r_{1}+\imath r_{2}\in\CC$. We may assume that
  \[
    P_{1}=P_{0}+r_{1}e^{\imath\theta},\quad P_{2}=P_{0}+Re^{\imath\theta},\quad P_{3}=P_{0}+\imath r_{2}e^{\imath\theta},
  \]
  for a suitable choice of an angle $\theta$.

  In order to center coordinates in $P_{0}$ and to align the legs of the rectangle to the coordinate axis, we perform the following change of coordinates:
  \[
    z\mapsto w(z)=e^{-\imath\theta}(z-P_{0}),
  \]
  whose inverse reads as $w\mapsto z(w)=e^{\imath\theta}w+P_{0}$. In the new coordinates, we have that
  \[
    w(P_{0})=0,\quad w(P_{1})=r_{1},\quad w(P_{2})=R,\quad w(P_{3})=\imath r_{2}.
  \]
  We then define the following polynomial:
  \[
    A(w)=B(z(w))=(1-P_{0})B_{0}+P_{0}B_{1}+we^{\imath\theta}(B_{1}-B_{0}),
  \]
  and of course $A(0)=B(P_{0}), A(r_{1})=B(P_{1}), A(R)=B(P_{2})$ and $ A(\imath r_{2})=B(P_{3})$.

  Having ``moved'' the rectangle vertices to $0, r_{1}, R,\imath r_{2}$, we now write $A$ in new coordinates as follows. We set
  \[
    \begin{cases}
      A_{0}=(1-P_{0})B_{0}+P_{0}B_{1},\\
      A_{1}=(1-P_{0}-Re^{\imath\theta})B_{0}+(P_{0}+Re^{\imath\theta})B_{1},
    \end{cases}
  \]
  so that
  \[
    A(w)=\frac{R-w}{R} A_{0}+\frac{w}{R}A_{1}.
  \]
  Thanks to the new coordinates of $A$, we have that $A(0)=A_{0}$ and $A(R)=A_{1}$. Notice that, if we start with $\Phi(z)=B(z)LB(z)^{c}$, then $\Phi(w)=A(w)LA(w)^{c}=\Phi(w(z))$.
\end{remark}

\begin{lemma}
  \label{lem:linear-relation}
  If $A(z) \in \HH_\CC[z]$ is of degree one and $P_0$, $P_1$, $P_2$, $P_3 \in
  \CC$ are the vertices of a rectangle (arranged in anticlockwise order) with
  legs $r_1$, $r_2 > 0$, then the values $\varphi_\ell \coloneqq
  A(P_\ell)LA^c(P_\ell)$, $\ell \in \{0,1,2,3\}$ satisfy the linear relation
  \eqref{eq2} with $R = r_1 + \imath r_2$.
\end{lemma}

\begin{proof}
  As explained in Remark~\ref{changeofvar}, up to an affine change of variables, we may assume that $P_{0}=0, P_{1}=r_{1}, P_{2}=R, P_{3}= \imath r_{2}$ and that $A(z)=\bigl(\frac{R-z}{R}\bigr)A_{0}+\frac{z}{R}A_{1}$. Then we get
  \begin{equation*}
    \begin{cases}
      \Phi(r_1) = \varphi_1 = \frac{1}{R^2}\bigl(
          (\imath r_2)^2A_{0}LA_{0}^c +
          \imath r_1r_2(A_{0}LA_1^c +
          A_1LA_0^c)+r_1^2A_1LA_1^c \bigr),\\
      \Phi(\imath r_2) = \varphi_3 = \frac{1}{R^{2}}\bigl(
          r_1^2A_0LA_0^c +
          \imath r_1r_2(A_0LA_1^c+A_1LA_0^c) +
          (\imath r_2)^2A_1LA_1^c\bigr).
    \end{cases}
  \end{equation*}
  By setting $A_{01}=A_0LA_1^c+A_1LA_0^c$ and using the two equalities
  $\varphi_0=A_0LA_0^c$ and $\varphi_2=A_1LA_1^c$, the last system can also be written
  as
  \begin{equation}
    \label{system1}
    \begin{cases}
      \imath r_1r_2A_{01} = -(\imath r_2)^2\varphi_0 + R^2\varphi_1 - r_1^2\varphi_2,\\
      \imath r_1r_2A_{01} = -r_1^2\varphi_0 + R^2\varphi_3 - (\imath r_2)^2\varphi_2.
    \end{cases}
  \end{equation}
  The difference of these two equations equals
  \begin{equation*}
    -(\imath r_{2})^{2}\varphi_{0}+R^{2}\varphi_{1}-r_{1}^{2}\varphi_{2}=
    -r_{1}^{2}\varphi_{0}+R^{2}\varphi_{3}-(\imath r_{2})^{2}\varphi_{2},
  \end{equation*}
  which can also be written as~eqref{eq2}.
\end{proof}

\begin{proof}[Proof of Theorem~\ref{maincub}]
  Equation \eqref{eq1} is necessary because of $\Phi(z) \Phi(z)^c = 0$, Equation~\eqref{eq2} is necessary by Lemma~\ref{lem:linear-relation} and Equation~\ref{eq3} is necessary because $\Phi(z)$ is a birational parametrization of the conic $\mathcal{N}$ and thus preserves cross ratios.

  Assume now that Equations~\eqref{eq1}, \eqref{eq2}, and \eqref{eq3} are fulfilled. The first and the last equation imply existence of a birational parametrization $\Phi(z)$ of $\mathcal{N}$ such that the interpolation conditions are satisfied in \emph{projective sense,} that is, for $\ell \in \{0,1,2,3\}$ there exist $\mu_\ell \in \CC$ such that $\mu_\ell \Phi(P_\ell) = \varphi_\ell$. Moreover, by Theorem~\ref{th:AA} there exists $A(z) \in \HH_\CC(z)$ such that $\Phi(z) = A(z)LA^c(z)$. Because $\deg\Phi(z) = 2$, we have $\deg A(z) = 1$ and, invoking Lemma~\ref{lem:linear-relation} once more, to conclude
  \begin{equation*}
    |R|^{2}\mu_0\varphi_{0} +
    R^{2}\mu_1\varphi_{1} -
    |R|^{2}\mu_2\varphi_{2} -
    R^{2}\mu_3\varphi_{3}=0.
  \end{equation*}
  Together with \eqref{eq2} this implies $\mu_0 = \mu_1 = \mu_2 = \mu_3$ and, by suitably scaling $A(z)$ we obtain exact equality $\Phi(P_\ell) = \varphi_\ell$ for $\ell \in \{0,1,2,3\}$.
\end{proof}

As far as actual computations are concerned, we observe that
\begin{equation}
  \label{eq:null-cone-parametrization}
  N\colon \CC^2 \to \mathcal{N},\quad
  (s, t) \mapsto
  (s^2+t^2)i + \imath(s^2-t^2)j + 2\imath stk
\end{equation}
parametrizes $\mathcal{N}$. In order to produce suitable input $\varphi_\ell$, $\ell \in \{0,1,2,3\}$, for a given rectangle, we can select parameter values $(s_\ell,t_\ell)$ subject to
\begin{equation*}
  \CR(z_0,z_1,z_2,z_3) = \CR((s_0,t_0), (s_1,t_1), (s_2,t_2), (s_3,t_3)),
\end{equation*}
set $\varphi_\ell \coloneqq \nu_\ell N(s_\ell,t_\ell)$ and compute $\nu_\ell$ such that \eqref{eq2} is satisfied.

Now, Equation~\eqref{eq3} implies existence of projective map
\begin{equation}
  \label{eq:projective-map}
  f\colon \PP^1(\CC) \to \PP^1(\CC)
\end{equation}
such that
\begin{equation*}
  f([z_0]) = [s_0 + \imath t_0],\
  f([z_1]) = [s_1 + \imath t_1],\
  f([z_2]) = [s_2 + \imath t_2],\
  f([z_3]) = [s_3 + \imath t_3].
\end{equation*}
A straightforward computation shows that the map
\eqref{eq:null-cone-parametrization} can be generated by the polynomial $A(s,t)
= ti + s$ as $N(s,t) = A(s + \imath t)LA^c(s + \imath t)$. By plugging
\eqref{eq:projective-map} into $A(s + \imath t)$ we obtain a suitable
quaternionic preimage $A(z) = A(u + \imath v)$.

\begin{example}
  \label{ex:1}
  We select the rectangular parameter domain with vertices
  \begin{equation*}
    P_0 = 0,\quad
    P_1 = 1,\quad
    P_2 = 1 + 2\imath,\quad
    P_3 = 2\imath
  \end{equation*}
  and the null vectors
  \begin{equation*}
    \begin{aligned}
      \varphi_0 &= \nu_0 N(1,0) = \nu_0(i + \imath j),\\
      \varphi_1 &= \nu_1 N(\imath, 1) = \nu_1(-2(\imath j + k)),\\
      \varphi_2 &= \nu_2 N(1, 2) = \nu_2(5i - 3\imath j + 4\imath k)
    \end{aligned}
  \end{equation*}
  with yet to be determined $\nu_0$, $\nu_1$, $\nu_2 \in \CC$.
  The cross ratio condition
  \begin{equation*}
    \CR(0, 1, 1 + 2\imath, 2\imath) = \CR([1,0], [\imath, 1], [1,2], [s_3,t_3])
  \end{equation*}
  yields $[s_3,t_3]=[5-2\imath,8]$ and results in
  \begin{equation*}
    \varphi_3 = \nu_3((85-20\imath)i + (20-43\imath)j + (32+80\imath)k)
  \end{equation*}
  with $\nu_3 \in \CC$. Now we use the linear relation \eqref{eq2} to
  determine the coefficients $\nu_0$, $\nu_1$, $\nu_2$, and
  $\nu_3$. We have $R = 1 + 2\imath$ and hence $\vert R \vert = \sqrt{5}$.
  This gives the equation
  \begin{multline*}
    5(i+\imath j)\nu_0
    +(3-4\imath)(2\imath j+2k)\nu_1
    -5(5i-3\imath j+4\imath k)\nu_2\\
    +(3-4\imath)((85-20\imath)i+(20-43\imath)j+(32+80\imath)k)\nu_3 = 0
  \end{multline*}
  with solution (up to scalar multiples)
  \begin{equation*}
    \nu_0 = 25,\quad
    \nu_1 = -16,\quad
    \nu_2 = 12-16\imath,\quad
    \nu_3 = 1.
  \end{equation*}
  Hence,
  \begin{gather*}
    \varphi_0 = 25i + 25\imath j,\quad
    \varphi_1 = 32\imath j + 32 k,\\
    \varphi_2 = (60-80\imath)i - (48+36\imath)j + (64+48\imath)k,\\
    \varphi_3 = (85-20\imath)i + (20-43\imath)j + (32+80\imath)k.
  \end{gather*}
  The projective transformation mapping relating $z_\ell$ and $s_\ell + \imath
  t_\ell$ for $\ell \in \{0,1,2,3,\}$ is given as
  \begin{equation*}
    s = 25 - 5z,\quad
    t = -20\imath z.
  \end{equation*}
  The quaternionic preimage is $A(z) = -(5+20\imath i)z + 25$. Transforming to
  a quaternionic polynomial yields
  \begin{equation*}
    A(z) = -(5+20j)z + 25.
  \end{equation*}
  With this, we obtain the minimal isothermal Enneper surface patch
  \begin{multline*}
    X(u,v) = \Re \int A(u+\imath v) L A^c(u+\imath v) \\=
    \frac{25}{3}
    \begin{pmatrix}
      -15u^3+45uv^2-15u^2+15v^2+75u\\
      -51u^2v+17v^3+30uv-75v \\
      -8u^3+24uv^2+60u^2-60v^2
    \end{pmatrix},\quad
    (u,v) \in [0,1] \times [0,2].
  \end{multline*}
  It is displayed in Figure~\ref{fig:1}. The partial derivative vectors in the
  vertices (scaled with a factor of $10$) are shown as well.
\end{example}

\begin{figure}
  \centering
  \includegraphics[]{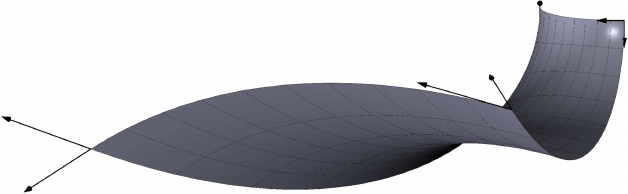}
  \caption{Enneper surface patch to Example~\ref{ex:1}.}
  \label{fig:1}
\end{figure}

\section{Outlook: Minimal Surfaces for CAGD}

Using the setup and insight gained in previous sections, using rational minimal surfaces in a CAGD context seems feasible. We believe that typical interpolation or approximations task can eventually be solved, possibly at the cost of high parametrization degrees. Challenges include control of corner points or boundary curves as well as shape control, quite similar to PH curves. We illustrate this at hand of a final example.

\begin{example}
	Notice that multiplying  an~isotropic complex curve $\Phi(z)$ by a complex function $\lambda(z)$ provides a new minimal surface with the same Gauss image. Indeed, recall that partial derivatives of $X(u,v)=\Re\int\lambda(z)\Phi(z)\dif  z$ are related to $\lambda(z)\Phi(z)$ via \eqref{eq:Rder} which directly implies
	
	\begin{equation*}
		\frac{\partial X}{\partial u}\times\frac{\partial X}{\partial v}=-\frac{\sqrt{-1}|\lambda|^2}{2}\left(\Phi\times\overline{\Phi}\right)
	\end{equation*}
	and we see that a function $\lambda$ affects only the magnitude of the~normal vector field. Nevertheless, one has to be cautious about the roots of $\lambda$ that correspond to the points on the surface, where the normal is not defined.  another way to see this is looking at the complex curve $\Psi(z)=\int\lambda(z)\Phi(z)\dif z$. If $\lambda(z_0)=0$ then $\Psi'(z_0)=0$ and the point is a local singularity on the curve. In the example depicted in Figure~\ref{fig:knots}, the~image of a small closed loop around $z_0$ is a knotted curve whose complexity depends on the multiplicity of the root.
	\begin{figure}
		\centering
		\includegraphics[width=0.28\textwidth]{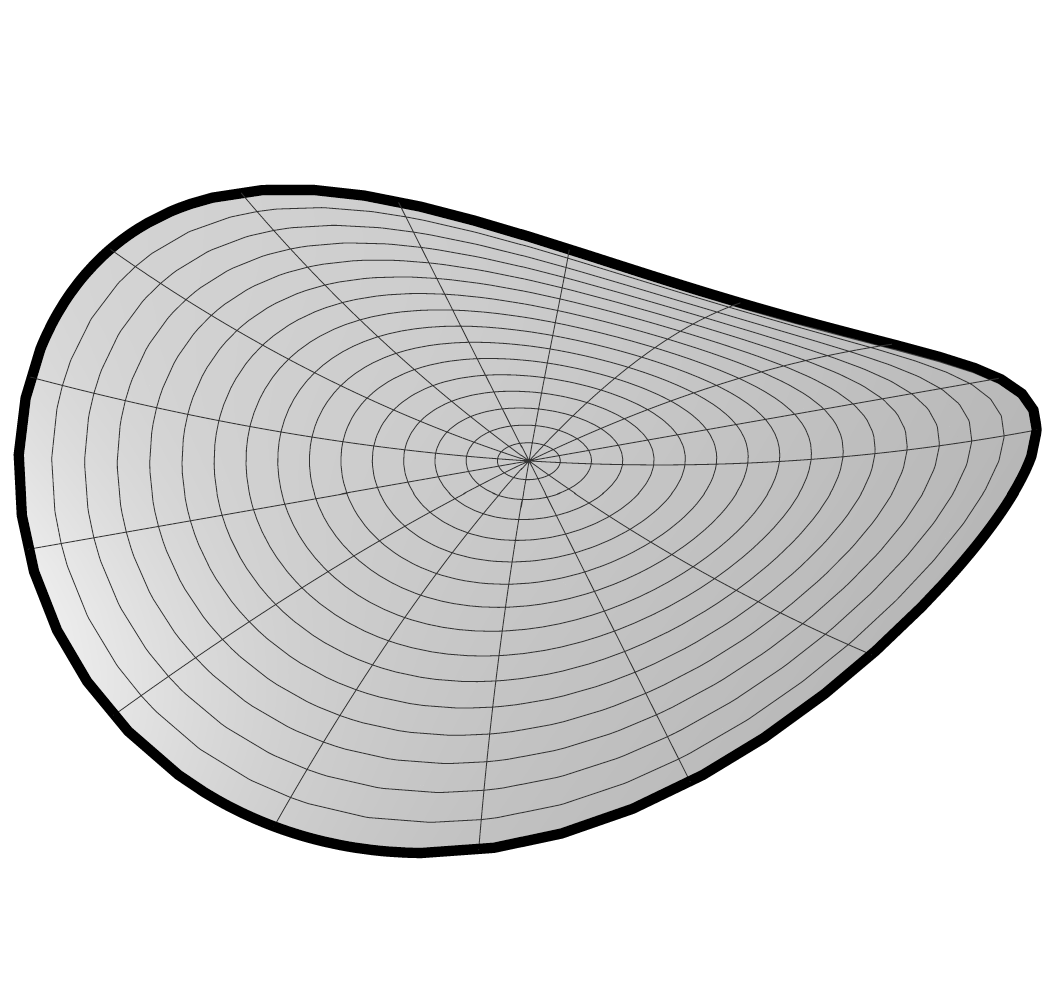}
		\hspace*{4ex}
		\includegraphics[width=0.28\textwidth]{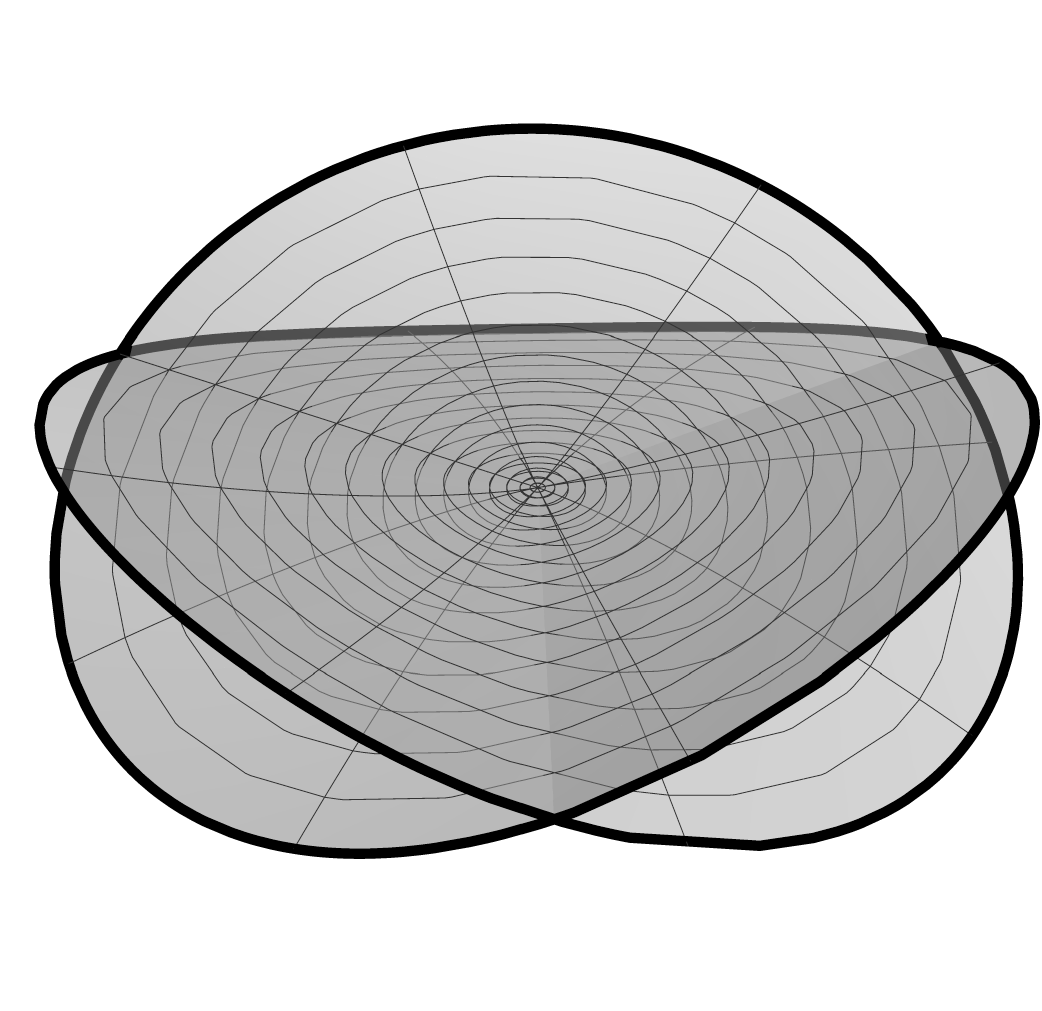}
		\hspace*{4ex}
		\includegraphics[width=0.28\textwidth]{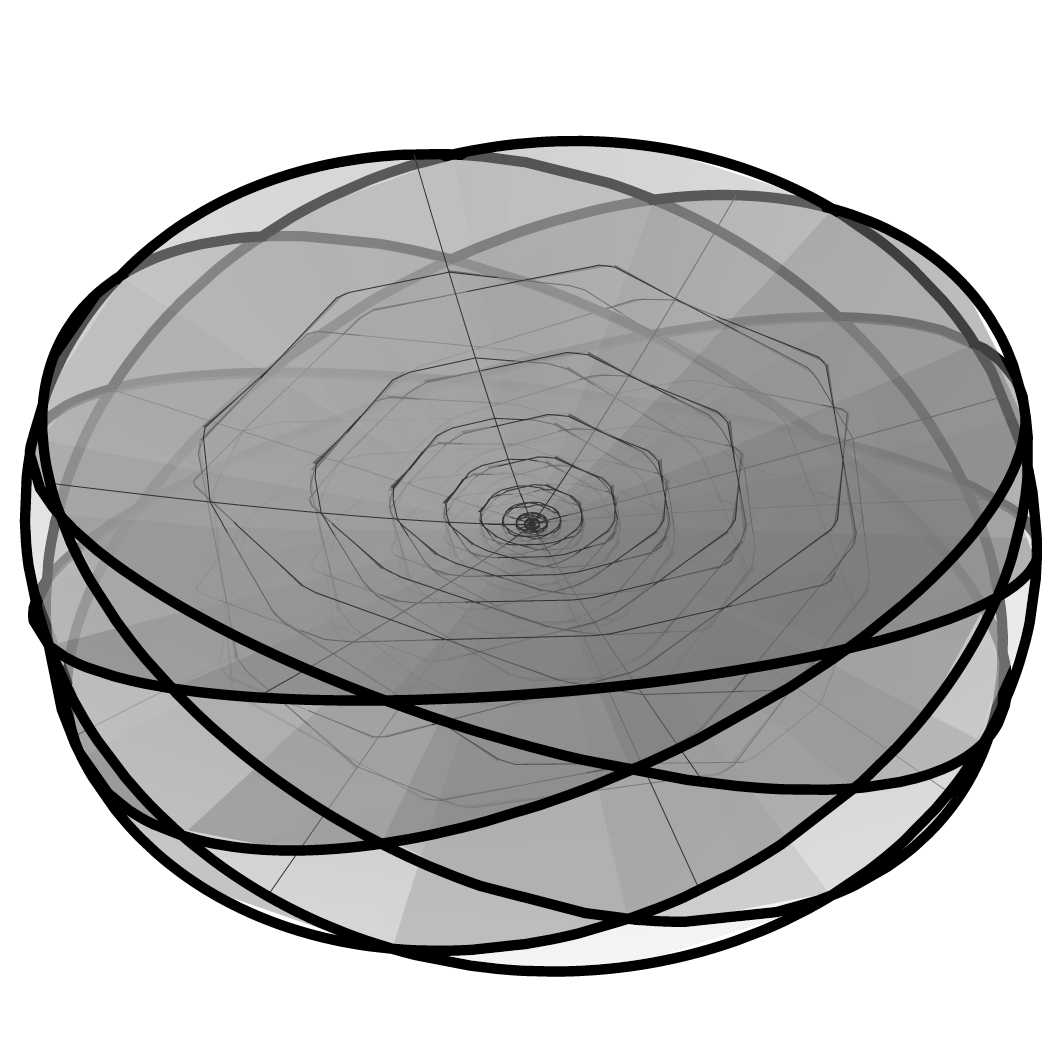}
		\caption{A local behavior of the minimal surface around the~root of $\lambda$ of multiplicity $0$, $1$ and $5$ respectively.}
		\label{fig:knots}
	\end{figure}
\end{example}

\bigbreak

\noindent{\bf Acknowledgments.}
This work stems from the collaboration between Amedeo Altavilla and Hans-Peter Schr\"ocker, who co-supervised Angela Teresa Masciale’s master's thesis. Masciale, supported by a scholarship, had the opportunity to conduct part of her research abroad, which led to Altavilla and Schr\"ocker meeting in response to her request.

The final part of Masciale’s thesis focused on parameterizations of surfaces that preserve the PH condition. An observation made in this section provided the foundation for some of the ideas developed in this work. For these reasons, the first and second author are grateful to Angela Teresa Masciale for her initiative and enthusiasm, which—directly or indirectly—contributed to making this research possible.

\bibliographystyle{amsplain}
\providecommand{\bysame}{\leavevmode\hbox to3em{\hrulefill}\thinspace}
\providecommand{\MR}{\relax\ifhmode\unskip\space\fi MR }
\providecommand{\MRhref}[2]{\href{http://www.ams.org/mathscinet-getitem?mr=#1}{#2}
}
\providecommand{\href}[2]{#2}

\end{document}